\documentclass{amsart}
\usepackage{amssymb,amsmath,latexsym,mathrsfs}
\usepackage{times}

\usepackage[dvips]{graphicx}

\numberwithin{equation}{section}

\newtheorem{thm}{Theorem}[section]
\newtheorem{dfn}[thm]{Definition}
\newtheorem{cor}[thm]{Corollary}

\newtheorem{prop}[thm]{Proposition}
\newtheorem{lemma}[thm]{Lemma}

\newcommand{\join}{\lor}
\newcommand{\meet}{\land}
\newcommand{\del}{\backslash}
\newcommand{\cl}{\mathrm{cl} }

\newcommand{\mcA}{\mathcal{A}}
\newcommand{\mcB}{\mathcal{B}}
\newcommand{\mcC}{\mathcal{C}}
\newcommand{\mcZ}{\mathcal{Z}}
\newcommand{\mcI}{\mathcal{I}}

\newcommand{\mcF}{\mathcal{F}}
\newcommand{\mcR}{\mathcal{R}}
\newcommand{\frp}{\mathbin{\Box}}

\title{Semidirect sums of matroids}

\author[J.~Bonin]{Joseph E.~Bonin} \address[J.~Bonin]
{Department of Mathematics\\ The George Washington University\\
  Washington, D.C. 20052, USA} \email[J.~Bonin] {jbonin@gwu.edu}
\author[J.P.S.~Kung] {Joseph P.S.~Kung} \address[J.P.S.~Kung]
{Department of Mathematics\\
  University of North Texas\\
  Denton, TX 76203, USA} \email[J.P.S.~Kung] {kung@unt.edu}
\date{\today}


\begin{document}

\begin{abstract}
  For matroids $M$ and $N$ on disjoint sets $S$ and $T$, a semidirect
  sum of $M$ and $N$ is a matroid $K$ on $S\cup T$ that, like the
  direct sum and the free product, has the restriction $K|S$ equal to
  $M$ and the contraction $K/S$ equal to $N$.  We abstract a matrix
  construction to get a general matroid construction: the matroid
  union of any rank-preserving extension of $M$ on the set $S\cup T$
  with the direct sum of $N$ and the rank-$0$ matroid on $S$ is a
  semidirect sum of $M$ and $N$.  We study principal sums in depth;
  these are such matroid unions where the extension of $M$ has each
  element of $T$ added either as a loop or freely on a fixed flat of
  $M$.  A second construction of semidirect sums, defined by a Higgs
  lift, also specializes to principal sums.  We also explore what can
  be deduced if $M$ and $N$, or certain of their semidirect sums, are
  transversal or fundamental transversal matroids.
\end{abstract}
\maketitle

\markboth{J.~Bonin and J.P.S.~Kung, \emph{Semidirect sums of
    matroids}}{J.~Bonin and J.P.S.~Kung, \emph{Semidirect sums of
    matroids}}
  
\section{Block upper-triangular matrices and semidirect
  sums}\label{sec:matrix} 

A simple way to combine two matrices $A$ and $B$ over a field
$\mathbb{F}$ is to have them be blocks of a block-diagonal matrix.  A
richer collection of matrices results by putting a third matrix $U$
over $\mathbb{F}$ in the upper right corner to obtain a block
upper-triangular matrix $(A,B;U)$, where
\begin{equation*}\label{eqn:matrixform}
  (A,B;U) = 
  \left(
    \begin{array}{cc}
      A \quad  &   U  \\
      0 \quad  &    B
    \end{array}
  \right).
\end{equation*}
In this paper, we consider the problem of extending this construction
to matroids.

The matroid represented by a block-diagonal matrix $(A,B;0)$ is a
direct sum.  Let $S$ and $T$ be disjoint sets, $M$ a matroid on $S$,
and $N$ a matroid on $T$.  The \emph{direct sum} $M \oplus N$ of $M$
and $N$ is the matroid on the union $S \cup T$ with rank function
given by
\[
r_{M \oplus N}(X) = r_M (X \cap S) +r_N (X \cap T)
\]
for $X \subseteq S \cup T$.  There is an equivalent way to define the
direct sum by restrictions and contractions: a matroid $K$ on $S \cup
T$ is the direct sum $M \oplus N$ if and only if
\begin{equation}\label{eqn:directsum}
K | S = K/T = M \quad \mathrm{and} \quad K |T = K /S = N.
\end{equation} 
Semidirect sums are defined by relaxing condition
(\ref{eqn:directsum}).

\begin{dfn}\label{dfn:SDP}
  Let $M$ and $N$ be matroids on the disjoint sets $S$ and $T$.  A
  matroid $K$ on the union $S \cup T$ is a \emph{semidirect sum} of
  $M$ and $N$ if the restriction $K|S$ is $M$ and the contraction
  $K/S$ is $N$.
\end{dfn}

Except in the special case of direct sums, the order of $M$ and $N$ in
a semidirect sum is important.  Semidirect sums are much more general
than direct sums.  In particular, if $K$ is a matroid on $E$ and
$\emptyset \subsetneq S \subsetneq E$, then $K$ is a semidirect sum of
the restriction $K|S$ and the contraction $K/S;$ thus, every matroid
with more than one element is a nontrivial semidirect sum.  Also, if
$r(N)=0$, then the semidirect sums of $M$ and $N$ are precisely the
extensions of $M$ to $S\cup T$ that have rank $r(M)$.  

We shall now show that Definition \ref{dfn:SDP} specializes to block
upper-triangular matrices in the case of matroids that are
representable over a given field.  Let $D[R|E]$ be a matrix with rows
indexed by $R$ and columns indexed by $E$.  The \emph{column matroid}
$K$ of $D$ is the matroid on $E$ defined by linear dependence of the
column vectors.  The column matroid remains unchanged under
nonsingular row operations on $D$.

It is well known that for $E_1 \subseteq E$, a matrix representation
of the contraction $K/E_1$ can be obtained as follows.  By permuting
the columns, we may assume that those in $E_1$ are left-most.  Let
$E_2$ be the set difference $E-E_1$.  Fix a basis $I$ of $K|E_1$.  Use
nonsingular row operations to transform $D$ into $D'$ so that the
submatrix $D'[R|I]$ consists of an $|I| \times |I|$ non-singular
matrix placed on top of a zero matrix.  Let $R_1$ be the first $|I|$
rows and let $R_2$ be the rest. The columns in $E_1$ are linearly
dependent on those in $I$, so $D'[R_2|E_1]$ is a zero matrix.  Thus,
\[
D'= \left(
\begin{array}{cc}
  D'[R_1|E_1] \quad  &   U  \\
  0 \quad  &    D'[R_2|E_2]
\end{array}
\right) \] for some matrix $U$.  The contraction $K/E_1$ is the column
matroid of the lower diagonal block $D'[R_2|E_2]$.
It follows that if the matrix rank of $A$ is its number of rows, then
the column matroid of the lower block $B$ in $(A,B;U)$ is the
contraction of the column matroid of $(A,B;U)$ by $S$.  These remarks
give the following result.

\begin{prop}\label{prop:representablecase} 
  Let $A$ and $B$ be matrices over a field $\mathbb{F}$, with disjoint
  sets $S$ and $T$ indexing the columns, with column matroids $M$ and
  $N$, and where $A$ has $r(M)$ rows.  The column matroid of any
  matrix $(A,B;U)$ over $\mathbb{F}$ is a semidirect sum of $M$ and
  $N$.  Conversely, if a semidirect sum $K$ of $M$ and $N$ can be
  represented as the column matroid of a matrix $D$ over $\mathbb{F}$,
  then $D$ can be transformed by nonsingular row operations into a
  matrix $(A,B;U)$ such that the column matroid of $A$ is $M$ and that
  of $B$ is $N$.
\end{prop}
  
The objective of this paper is to study constructions of semidirect
sums of matroids that use only the matroid structure.  In Section
\ref{sec:Union}, we discuss matroid unions and then use this operation
to construct semidirect sums.  A principal sum of $M$ and $N$ is the
matroid union of an extension of $N$ by loops with an extension of $M$
in which each element of $N$ is added either as a loop or freely on a
fixed flat of $M$; Section \ref{sec:princsums} contains a detailed
analysis of these special semidirect sums.  In Section
\ref{sec:othercon}, we treat a construction of semidirect sums that is
defined using a Higgs lift and that gives another approach to
principal sums.  In the final section, we show that if a semidirect
sum of $M$ and $N$ given by the matroid union construction is
transversal, then $M$ and $N$ are also transversal; the counterpart
holds for fundamental transversal matroids and, with additional
hypotheses, the converses also hold.

We assume a working knowledge of matroid theory as described in
\cite{oxley}.  We also use more specialized results from the theory of
matroid unions, quotients, the weak order, Higgs lifts, and
transversal matroids; these results will be briefly summarized where
they are needed and the reader is referred to \cite{Brylawski, Strong,
  Weak, oxley, Welsh} for detailed accounts.  We reserve calligraphic
fonts for collections of sets, such as the collections $\mcI(M)$,
$\mcC(M)$, and $\mcF(M)$ of independent sets, circuits, and flats of a
matroid $M$.  We abbreviate a single-element set $\{a\}$ by $a$.  We
use $U_{r,E}$ when we want to specify the set on which the rank-$r$
uniform matroid $U_{r,n}$ is defined.

We close this introductory section with two basic observations, the
second of which follows from the fact that contraction and deletion
are dual operations.

\begin{prop}\label{prop:rkdl}
  If $K$ is a semidirect sum of the matroids $M$ and $N$, then
  \begin{enumerate}
  \item $r(K) = r(M)+r(N)$ and
  \item the dual $K^*$ is a semidirect sum of the duals, $N^*$ and
    $M^*$.
  \end{enumerate}
\end{prop}

\section{Semidirect sums from matroid unions}\label{sec:Union}  

Nash-Williams introduced matroid unions in \cite{NW}.  We begin by
recalling some basic facts about this operation; further information
can be found in \cite{oxley,Welsh}.

Let $G$ and $H$ be matroids, both on the set $E$.  The \emph{matroid
  union} $G \vee H$ is the matroid on $E$ whose collection of
independent sets is
$$\mcI(G\join H) = \{I_G\cup I_H\,:\,I_G\in\mcI(G) \text{ and }
I_H\in\mcI(H)\}.$$ In words, a set is independent in $G \vee H$ if and
only if it is a union of a $G$-independent set and an $H$-independent
set.  Clearly $G\join H = H\join G$.  Matroid unions can also be
defined by the rank function: for $X \subseteq E$,
\begin{equation}\label{eqn:rankunion}
  r_{G \vee H} (X) = \min_{W \subseteq X} \,\,
  \{r_G(W) + r_H(W) + |X - W|\}.   
\end{equation} 

A \emph{quotient} of a matroid $L$ on a set $E$ is a matroid $Q$ on
the same set such that for all subsets $F \subseteq E$, the closure
$\cl_L(F)$ of $F$ in $L$ is contained in its closure, $\cl_Q(F)$, in
$Q$.  When $Q$ is a quotient of $L$, we also say that $L$ is a
\emph{lift} of $Q$.  If $L$ is the column matroid of $\mathrm{Mat}(L)$
and if a row is removed from $\mathrm{Mat}(L)$, then the column
matroid of the smaller matrix is a quotient of $L$.  We will use two
properties of quotients: if $Q$ is a quotient of $L$, then
$\mcF(Q)\subseteq \mcF(L)$ (that is, every $Q$-flat is an $L$-flat)
and every $L$-circuit is a union of $Q$-circuits. (See, for example,
\cite[Proposition 8.1.6c]{Strong}.)

\begin{lemma}\label{lem:quotient} 
  The matroids $G$ and $H$, both on the set $E$, are quotients of
  their matroid union $G \vee H$.
\end{lemma}

\begin{proof}  
  We show that $\cl_{G \vee H}(F) \subseteq \cl_G(F)$ for all subsets
  $F \subseteq E$.  Suppose that $a \not\in F$ but $a \in \cl_{G \vee
    H}(F)$.  Thus, there is a $(G \vee H)$-independent set $I$ such
  that $I \subseteq F$ and $I \cup a$ is $(G \vee H)$-dependent.  Now
  $I = I_G \cup I_H$ for some sets $I_G\in\mcI(G)$ and
  $I_H\in\mcI(H)$.  If $I_G \cup a$ were $G$-independent, then taking
  its union with $I_H$ would show that $I \cup a \in \mcI( G \vee H)$.
  This contradiction gives $I_G \cup a\not\in\mcI(G)$, so $a \in
  \cl_G(F)$.
\end{proof}  

Since intersections of flats are flats, we have the following
consequence of Lemma \ref{lem:quotient}.

\begin{cor}\label{cor:flats}
  If $U$ is a $G$-flat and $V$ is an $H$-flat, then $U$, $V$, and $U
  \cap V$ are $(G \vee H)$-flats.
\end{cor} 

The following example shows that Corollary \ref{cor:flats} does not
describe all flats of a matroid union.  Consider the rank-$2$ column
matroid $M$ on the set $\{a,b,c,d,e\}$ given by the matrix
\[
\left(
\begin{array}{ccccc} 
  x_a & x_b & x_c & 0 & 0
  \\
  1 & 0 & 0 & 0 & 1
\end{array}
\right).
\]
This is the matroid union of two rank-$1$ matroids: in the first, $d$
and $e$ are the loops, so its flats are $\{d,e\}$ and $\{a,b,c,d,e\}$;
in the second, $b$, $c$, and $d$ are the loops, so its flats are
$\{b,c,d\}$ and $\{a,b,c,d,e\}$.  The flats of $M$ are $\{d\}$,
$\{b,c,d\}$, $\{a,d\}$, $\{d,e\}$, and $\{a,b,c,d,e\}$.  The $M$-flat
$\{d\}$ is the intersection $\{d,e\} \cap \{b,c,d\}$.  The flat
$\{a,d\}$ is not described in Corollary \ref{cor:flats}.

\begin{cor}\label{cor:cyclicsets}
  If $C$ is a $(G \join H)$-circuit, then $C$ is a union of
  $G$-circuits as well as a union of $H$-circuits.  In particular, if
  a set is cyclic (that is, a union of circuits) in $G \join H$, then
  it is cyclic in both $G$ and $H$.
\end{cor}

The converse of Corollary \ref{cor:cyclicsets}, as one might expect,
is false.  For example, if $0<r<n$, the uniform matroids $U_{r,n}$ and
$U_{n-r,n}$ have many cyclic sets in common, but their matroid union
$U_{r,n} \vee U_{n-r,n}$, which is $U_{n,n}$, has no circuits and
hence no nonempty cyclic sets.

We next address deletions and certain contractions of matroid unions.

\begin{lemma}\label{lem:uniondel}
  Let $G$ and $H$ be matroids on $E$.  For any subset $X$ of $E$,
  $$(G\join H)\del X = (G\del X)\join (H\del X).$$
  If each element of $X$ is a loop of at least one of $G$ or $H$,
  then
  $$(G\join H)/X = (G/X)\join (H/X).$$
\end{lemma}

\begin{proof}
  The first part is immediate.  For the second, note first that if $x$
  is a loop of both $G$ and $H$, then it is a loop of $G \join H$, so
  $(G \join H)/x = (G \join H) \del x$.  To complete the proof, it
  suffices to treat a single-element contraction $(G\join H)/x$ where
  $x$ is a loop of $H$ but not of $G$.  In this case the result holds
  since the statements below are equivalent:
  \begin{enumerate}
  \item[(i)] $I\in\mcI((G\join H)/x)$,
  \item[(ii)] $I\cup x\in\mcI(G\join H)$,
  \item[(iii)] $I\cup x = (I_G\cup x)\cup I_H$ for some $I_G\cup
    x\in\mcI(G)$ and $I_H\in\mcI(H)$,
  \item[(iv)] $I= I_G\cup I_H$ for some $I_G\in\mcI(G/x)$ and
    $I_H\in\mcI(H/x)$,
  \item[(v)] $I\in\mcI((G/x)\join (H/x))$.\qedhere
\end{enumerate}
\end{proof}

To see that the hypothesis in the second part of Lemma
\ref{lem:uniondel} is needed, take $G$ and $H$ to be the uniform
matroid $U_{2,4}$ and let $X$ be a single-element set.

The set of all matroids on a given set $E$ is ordered by the
\emph{weak order}, denoted by $\leq_w$, where $H \leq_w G$ if and only
if $r_H(X)\leq r_G(X)$ for all subsets $X$ of $E$; equivalently, every
set that is independent in $H$ is also independent in $G$.  This
relation makes precise the idea that $G$ is freer than $H$.  The next
lemma follows easily from the definitions.

\begin{lemma}\label{lem:unionspreserveweak}
  Let $G_1$, $G_2$, $H_1$, $H_2$, $G$, and $H$ be matroids on the same
  set $E$.
  \begin{enumerate}
  \item If $G_1\leq_w G_2$ and $H_1\leq_w H_2$, then $G_1\join
    H_1\leq_w G_2\join H_2$.
  \item Assume $G_1\leq_w G\leq_w G_2$ and $H_1\leq_w H\leq_w H_2$. If
    $G_1\join H_1= G_2\join H_2=K$, then $G\join H =K$.
  \end{enumerate}
\end{lemma}

To understand matroid union intuitively, we look at the case when $G$
and $H$ are the column matroids of the matrices $\mathrm{Mat}(G)$ and
$\mathrm{Mat}(H)$, both over the field $\mathbb{F}$. Let $K$ be the
column matroid of the matrix
\[
\left[\begin{array}{c}
\underline{\, \mathrm{Mat}(G) \,}\\
\, \mathrm{Mat}(H) \,
\end{array}\right]
\]
obtained by putting $\mathrm{Mat}(G)$ atop $\mathrm{Mat}(H)$.  It is
not hard to see (using, for example, the multiple Laplace expansion
for determinants) that if a set $I$ of columns is $K$-independent,
then $I$ is the union of a $G$-independent set $I_G$ and an
$H$-independent set $I_H$.  The converse is not necessarily true, as
there may be algebraic relations between the entries of
$\mathrm{Mat}(G)$ and $\mathrm{Mat}(H)$ affecting the linear
dependence of $I_G \cup I_H$.  The matroid union $G \vee H$ is
obtained by destroying these algebraic relations.  Let
$\mathbb{F}(x_e)$ be the transcendental extension of $\mathbb{F}$
obtained by adjoining elements $x_e$ transcendental over $\mathbb{F}$,
one for each element $e$ in $E$.  Let $\mathrm{GenMat}(G)$ be the
``generic'' matrix obtained by multiplying the column indexed by $e$
in $\mathrm{Mat}(G)$ by the transcendental $x_e$.  The matroid union
$G \vee H$ is the column matroid of the matrix
\[
\left[\begin{array}{c}
\underline{\, \mathrm{GenMat}(G) \,}\\
\, \mathrm{Mat}(H) \,
\end{array}\right]. 
\]

Matroid union is a matroid construction analogous to putting a generic
matrix on top of another matrix.  In particular, taking the matroid
union of a rank-preserving extension of $M$ and the extension of $N$
by loops is analogous to constructing a block upper-triangular matrix
where the upper blocks are generic submatrices.  This analogy is
formalized by the next theorem and illustrated by the example given in
Figure~\ref{whirl}.

\begin{thm}\label{unionthm}
  Assume that $M$ and $N$ are matroids on the disjoint sets $S$ and
  $T$.  Let $N_0$ be $N\oplus U_{0,S}$.  If $M^+$ is any extension of
  $M$ to $S\cup T$ with $r(M)=r(M^+)$, then the matroid union $M^+
  \join N_0$ is a semidirect sum of $M$ and $N$.
\end{thm}

\begin{proof}
  Lemma \ref{lem:uniondel} gives $(M^+ \join N_0)|S = M \join
  U_{0,S}$, which is $M$ since $U_{0,S}$ contains only loops.  The
  lemma also gives $(M^+ \join N_0)/S = (M^+/S) \join (N_0/S)$, which
  is $N$ since $M^+/S$ contains only loops.
\end{proof}

\begin{figure}
  \begin{center}
    \includegraphics[width = 3.15truein]{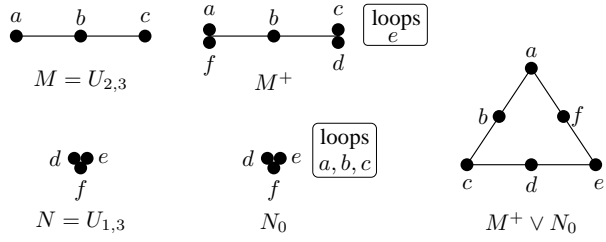}
 \end{center}
 \caption{The $3$-whirl as the matroid union of extensions of
   $U_{2,3}$ and $U_{1,3}$.}\label{whirl}
\end{figure}

The next lemma describes a natural way in which a circuit of $N$
extends to a circuit of $M^+ \join N_0$. For a basis $B$ of $M$ and $a
\in (S\cup T)- B$, let $C(a,B)$ be the \emph{fundamental circuit} in
$M^+$ of $a$ relative to $B$, that is, the unique circuit contained in
$B \cup a$. For a basis $B$ of $M$ and circuit $C$ of $N$, define
\begin{equation}\label{IBC}
  I(B,C) = \bigcup_{a\in C}\bigl(C(a,B)-a\bigr).
\end{equation}

\begin{lemma}\label{lem:Ncircuits}
  If $C\in\mcC(N)$ and $I\in\mcI(M)$, then
  \begin{enumerate}
  \item $I \cup C \in \mcI (M^+ \join N_0)$ if and only if $I \cup
    a \in \mcI(M^+)$ for some $a \in C$, and
  \item $I \cup C\in\mcC(M^+ \vee N_0)$ if and only if $I = I(B,C)$
    for some basis $B$ of $M$.
  \end{enumerate}
\end{lemma}

\begin{proof}  
  If $I \cup a$ is in $\mcI(M^+)$ for some $a \in C$, then, since $C -
  a$ is in $\mcI(N_0)$, their union, $I \cup C$, is in $\mcI(M^+ \join
  N_0)$.  Conversely, if $I \cup C \in \mcI(M^+ \join N_0)$, then $I
  \cup C= J_{M^+} \cup J_{N_0}$ for some sets $J_{M^+} \in \mcI(M^+)$
  and $J_{N_0} \in \mcI(N_0)$.  Elements in $I$ are loops of $N_0$ and
  $C \not\in \mcI(N_0)$, so $J_{N_0} \subsetneq C$.  Therefore $I
  \subsetneq J_{M^+}$, that is, there is an element $a \in C$ with $I
  \cup a \in \mcI(M^+)$. This verifies part (1).

  Moving to part (2), let $I = I(B,C)$, where $B$ is a basis of $M$.
  For all $a\in C$, the set $I$ contains $C(a,B)-a$, so $I \cup
  a\not\in \mcI (M^+)$. By part (1), $I \cup C \not\in \mcI(M^+ \join
  N_0)$. We will show that $I \cup C$ is a circuit by showing that
  each of its one-element deletions is independent.  If $b \in I$,
  then $b \in C(a,B)$ for some $a \in C$ and so $(I - b) \cup a
  \in \mcI(M^+)$; by part (1), $(I - b) \cup C \in \mcI (M^+ \join
  N_0)$.  If $a \in C$, then $(I \cup C) - a \in \mcI (M^+ \join N_0)$
  since $C-a\in\mcI(N_0)$ and $I\in\mcI(M^+)$.
   
  To finish the proof, let $I \cup C$ be an $(M^+ \join N_0)$-circuit.
  Extend $I$ to a basis $B$ of $M$.  By part (1), $I \cup a \notin
  \mcI (M^+)$ for all $a \in C$, so $C(a,B)-a \subseteq I$.
  Thus, $I(B,C) \subseteq I$.  By the previous paragraph, $I(B,C) \cup
  C$ is an $(M^+ \join N_0)$-circuit; since it is contained in the
  $(M^+ \join N_0)$-circuit $I\cup C$, we conclude that $I = I(B,C)$.
\end{proof}

To close this section, we note that semidirect sums can also be
obtained via the operation that is dual to matroid union.  For
matroids $G$ and $H$ defined on the same set $E$, their \emph{matroid
  intersection}, denoted $G \wedge H$, is $(G^* \vee H^*)^*$.  The
name comes from the fact that a subset $S$ of $E$ is spanning in
$G\meet H$ if and only if $S=S_G \cap S_H$ for some sets $S_G$ and
$S_H$ with $\cl_G(S_G)=E= \cl_H(S_H)$.  Proposition \ref{prop:rkdl}
and Theorem \ref{unionthm} give the next result.

\begin{thm}\label{wedgethm}
  Assume that $M$ and $N$ are matroids on the disjoint sets $S$ and
  $T$.  Let $M_1$ be $M\oplus U_{|T|,T}$.  If $N'$ is any coextension of
  $N$ to $S\cup T$ with $r(N')=r(N)+|S|$, then the matroid
  intersection $M_1 \wedge N'$ is a semidirect sum of $M$ and $N$.
\end{thm}

\section{Principal sums}\label{sec:princsums}

In this section we investigate the special case when $M^+$ is obtained
from a principal extension of $M$ by adding loops.

Let $S$ and $T$ be disjoint sets, fix subsets $A \subseteq S$ and $B
\subseteq T$, and let $M$ be a matroid on $S$.  Informally, the
matroid $M^+(A,B)$ is the extension of $M$ to $S\cup T$ constructed by
putting each element in $B$ freely on the $M$-flat spanned by $A$ and
adding each element of $T-B$ as a loop.  We will define $M^+(A,B)$
formally by iterated single-element principal extensions.

Let $K$ be a matroid on a set $E$, fix a subset $A$ of $E$, and let
$b$ be an element not in $E$.  The single-element extension $K +_A b$
is the matroid on $E \cup b$ with the rank function $r$ defined as
follows: for a subset $X$ of $E$, set $r(X)=r_K(X)$ and
\[
r(X \cup b) = \left\{
  \begin{array}{ll}
    r_K (X),  & \text{if}\,\, A \subseteq \cl_K(X),\\
    r_K (X) + 1, & \text{otherwise.}  
\end{array}\right.  
\]   
It is easy to check that $K +_A b$ is a matroid.  The inclusion $A
\subseteq \cl_K(X)$ is equivalent to the equality $r_K(X) = r_K(X\cup
A)$, so the rank function $r$ can be recast as follows: for
$X\subseteq E$ and $Y$ is a subset of the one-element set $b$,
\begin{equation}\label{eq:rkprext1}
  r(X\cup Y) = \min\{r_K(X\cup A),\, r_K(X)+|Y\cap b|\}.
\end{equation}

Now order the elements $b_1,b_2, \ldots, b_k$ in $B$ and define
$M^+(A,B)$ to be
\[
(((M +_A b_1) +_A b_2) +_A\cdots +_A b_k) \oplus U_{0,T-B}.
\]
A routine induction starting with equation (\ref{eq:rkprext1}) gives the first
assertion in the next lemma.  The second assertion follows easily from
the first.

\begin{lemma}\label{lem:PrincipalExtension}
  The rank function $r$ of the extension $M^+(A,B)$ of $M$ is given by
  \begin{equation}\label{eq:rkprext2}
    r(X\cup Y) = \min\{r_M(X\cup A),\, r_M(X)+|Y\cap B|\}
  \end{equation}
  for $X\subseteq S$ and $Y\subseteq T$.  The subsets of $S \cup T$
  that are independent in $M^+(A,B)$ are the unions $I \cup J$ where
  $I \in\mcI(M)$, $J \subseteq B$, and $r_M (I \cup A) - |I| \geq
  |J|$.
\end{lemma}

Note the geometry behind the second part of the lemma: $r_M(A\cup I) -
|I|$ elements are needed to extend $I$ to a basis of $\cl_{M^+}(A\cup
I)$, and since the elements of $B$ are added freely to $\cl_M(A)$, any
subset $J$ of $B$ of that size or smaller can be part of such a basis.

The semidirect sums in the next definition are our main objects of
study.

\begin{dfn}\label{dfn:principalsum}
  Let $S$ and $T$ be disjoint sets, fix subsets $A$ of $S$ and $B$ of
  $T$, and let $M$ be a matroid on $S$ and $N$ a matroid on $T$. The
  \emph{principal sum} $(M,N;A,B)$ is defined by
  \[
  (M,N;A,B) = M^+(A,B) \vee N_0.
  \]
\end{dfn}

Figure \ref{prtrans} illustrates this construction.  The special case
$(M,N;\emptyset,\emptyset)$ is the direct sum $M \oplus N$.  By
Theorem \ref{unionthm}, $(M,N;A,B)$ is a semidirect sum of $M$ and
$N$.

\begin{figure}
  \begin{center}
    \includegraphics[width = 3.1truein]{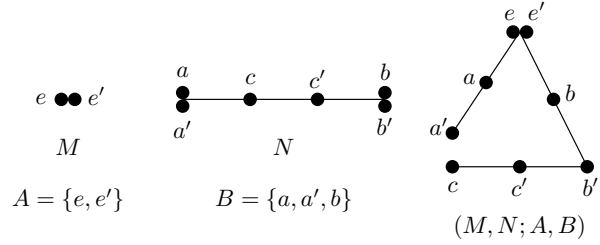}
  \end{center}
  \caption{A principal sum.}\label{prtrans}
\end{figure}

\begin{lemma}\label{lem:SDPindependent} 
  The independent sets of the principal sum $(M,N;A,B)$ are the unions
  of three disjoint sets, $I$, $D$, and $D'$, where \emph{(1)}
  $I\in\mcI(M)$, \emph{(2)} $D\in\mcI(N)$, and \emph{(3)} $D'\subseteq
  B$ with $ |D'| \leq r_M (I \cup A) - |I|$.
\end{lemma}

\begin{proof}  
  By Lemma \ref{lem:PrincipalExtension}, the set $I \cup D^{\prime}
  \cup D $ is the union of the $M^+(A,B)$-independent set $I\cup D'$
  and the $N_0$-independent set $D$, and so is independent in
  $(M,N;A,B)$.
\end{proof}

\begin{cor}
  If $A\subseteq A'\subseteq S$ and $B\subseteq B'\subseteq T$, then
  $(M,N;A,B)\leq_w (M,N;A',B')$.
\end{cor}

The principal sum $(M,N;S,T)$ is the free product $M \frp N$, which
was defined by Crapo and Schmitt \cite{fpm,uft}.  Indeed, by Lemma
\ref{lem:SDPindependent}, the independent sets of $(M,N;S,T)$ are the
sets $I \cup D$ where $I \in\mcI(M)$ and $D \subseteq T$ with
$|D|-r_N(D) \leq r(M) - |I|$, which is the description of the
independent sets of $M \frp N$ given in Proposition 1 of \cite{fpm}.
We remark that \cite[Proposition 7.2]{uft} says that the collection of
all semidirect sums of $M$ and $N$ is the interval $[M \oplus N, M
\frp N]$ in the weak order on the set of all matroids on $S \cup T$;
in other words, a matroid $K$ on $S\cup T$ is a semidirect sum of $M$
and $N$ if and only if
\[
M\oplus N\leq_w K\leq_w M\frp N.  
\]

Crapo and Schmitt \cite{fpm} used the free product to prove Welsh's
conjecture that $$f(m+n)\geq f(m)f(n)$$ where $f(k)$ is the number of
non-isomorphic matroids on a $k$-element set.  Independently and
simultaneously, Lemos \cite{Lemos} also proved this conjecture.
Lemos' proof used the $2$-sum $(M+ e)\oplus_2(N\times e)$ of the free
extension of $M$ by $e$ and the free coextension of $N$ by $e$.  With
Theorem \ref{thm:rank} and \cite[Exercise 7.1.1 (a)]{oxley}, it is not
hard to show that $(M+ e)\oplus_2(N\times e)$ is obtained from the
principal sum $(M+e,N;e,T)$ by deleting $e$.

With Lemma \ref{lem:SDPindependent}, we now derive a formula for the
rank function of a principal sum.

\begin{thm}\label{thm:rank}
  Let $P$ be the principal sum $(M,N;A,B)$.  For any sets $X\subseteq
  S$ and $Y\subseteq T$, the rank $r_P(X\cup Y)$ is the minimum of the
  following two quantities:
  \begin{equation}\label{ranka}
    r_M(X\cup A)+r_N(Y)
  \end{equation}
  \begin{equation}\label{rankb}
    r_M(X)+r_N(Y-B)+|Y\cap B|.
  \end{equation}
\end{thm}

\begin{proof}
  Equation (\ref{eqn:rankunion}) implies that for all subsets $W$ of
  $X\cup Y$,
  $$r_P(X\cup Y)\leq r_{M^+(A,B)}(W) + r_{N_0}(W) +|(X\cup Y)-W|.$$ 
  To show that expression (\ref{rankb}) is an upper bound on
  $r_P(X\cup Y)$, take $W=X\cup (Y-B)$ and use Lemma
  \ref{lem:PrincipalExtension}.  A similar argument using $W =X\cup
  A\cup Y$ shows that expression (\ref{ranka}) is an upper bound on
  $r_P(X\cup A\cup Y)$, and so on $r_P(X\cup Y)$.  Thus, it suffices
  to construct a subset of $X\cup Y$ in $\mcI(P)$ whose size is the
  minimum of expressions (\ref{ranka}) and (\ref{rankb}).  Note that
  the second is the minimum if and only if
  \begin{equation}\label{ineq:rkpivot}
    r_M(X\cup A)-r_M(X)\geq |Y\cap B| -r_N(Y) + r_N(Y-B).
  \end{equation}
  Let $I$ be a basis of $M|X$, so $|I| = r_M(X)$.  Let $D_0$ be a
  basis of $N|(Y-B)$, extend $D_0$ to a basis $D$ of $Y$, and let $D'
  = (Y\cap B)-D$.  Thus, $$|D'| =|Y\cap B| - \bigl(r_N(Y) -
  r_N(Y-B)\bigr).$$ If inequality (\ref{ineq:rkpivot}) holds, then
  $r_M(I\cup A)-|I|\geq |D'|$, so $I\cup D\cup D'\in\mcI(P)$ by Lemma
  \ref{lem:SDPindependent}, and this case is completed by observing
  that
  $$|I\cup D\cup D'|
  = r_M(X) + r_N(Y)+ \bigl(|Y\cap B| -r_N(Y) + r_N(Y-B)\bigr),$$ which
  simplifies to expression (\ref{rankb}).  If inequality
  (\ref{ineq:rkpivot}) fails, then let $D''$ be a subset of $D'$ of
  size $r_M(I\cup A)-|I|$, that is, $r_M(X\cup A)-r_M(X)$.  Thus,
  $I\cup D\cup D''\in\mcI(P)$ and
  $$|I\cup D\cup D''| = r_M(X) + r_N(Y)+ \bigl(r_M(X\cup
  A)-r_M(X)\bigr),$$ which simplifies to expression (\ref{ranka}),
  thereby completing the proof.
\end{proof}

The following special case will be used frequently.

\begin{cor}\label{cor:samerkasds}
  Let $P$ be the principal sum $(M,N;A,B)$.  For any sets $X\subseteq
  S$ and $Y\subseteq T$, if $A\subseteq X$ or $Y\cap B = \emptyset$,
  then $r_P(X\cup Y) = r_{M\oplus N}(X\cup Y) = r_M(X) + r_N(Y)$.
\end{cor}

An \emph{(order) ideal} of subsets of a set $E$ is a collection $\mcA$
of subsets of $E$ such that if $X\in \mcA$ and $Y\subseteq X$, then
$Y\in \mcA$.  The order-theoretic dual of an ideal is a \emph{filter},
that is, a collection $\mcB$ of subsets of $E$ such that if $X\in
\mcB$ and $X \subseteq Y$, then $Y \in \mcB$.

\begin{lemma}\label{lem:idealfilter}
  Partition the collection of subsets of $S\cup T$ into the following
  three sets (one or two of which might be empty), where we take
  $X\subseteq S$ and $Y\subseteq T$:
  \begin{align*}
    & \mcR_< = \{X\cup Y\,:\, r_M(X\cup A)+r_N(Y)<
    r_M(X)+r_N(Y-B)+|Y\cap B|\},
    \\
    & \mcR_= = \{X\cup Y\,:\, r_M(X\cup A)+r_N(Y)=r_M(X)+r_N(Y-B)+
    |Y\cap B|\},
    \\
    & \mcR_> = \{X\cup Y\,:\, r_M(X\cup A)+r_N(Y)>r_M(X)+r_N(Y-B)+
    |Y\cap B|\}.
  \end{align*}
  Also, set $\mcR_\leq = \mcR_<\cup\mcR_=$ and $\mcR_\geq =
  \mcR_=\cup\mcR_>$.  The collections $\mcR_<$ and $\mcR_\leq$ are
  filters and the collections $\mcR_>$ and $\mcR_\geq$ are ideals.
\end{lemma}

\begin{proof}
  The lemma is a consequence of inequalities (\ref{idealeqn1}) and
  (\ref{idealeqn2}) that we derive below.  If $X\subseteq X'\subseteq
  S$, then $r_{M/X'}(A-X')\leq r_{M/X}(A-X)$, that is,
  \begin{equation}\label{idealeqn1}
    r_M(X'\cup A)-r_M(X')\leq r_M(X\cup A)-r_M(X).
  \end{equation}
  For $Y\subseteq Y'\subseteq T$, the same type of argument applied to
  $S-Y'$ and $S-Y$ in $N^*$ gives
  $$r_{N^*}((S-Y)\cup B)-r_{N^*}(S-Y)\leq r_{N^*}((S-Y')\cup
  B)-r_{N^*}(S-Y'),$$ or, using the connection between the rank
  functions of $N$ and $N^*$,
  \begin{equation}\label{idealeqn2}
    |Y\cap B| - r_N(Y) + r_N(Y-B) \leq |Y'\cap B| - r_N(Y') +
    r_N(Y'-B).\qedhere 
  \end{equation}
\end{proof}

The next result stands in contrast to general matroid unions since
duals of matroid unions need not be matroid unions.

\begin{thm}\label{thm:dual}
  The dual of a principal sum is a principal sum; specifically,
  $$(M,N;A,B)^*=(N^*,M^*;B,A).$$
\end{thm}

\begin{proof}
  For $X\subseteq S$ and $Y\subseteq T$, the rank of $X\cup Y$ in
  $(M,N;A,B)^*$ is the minimum of the following two expressions:
  \begin{equation}\label{dual1}
    |X|+|Y| -r(M)-r(N) + r_M((S-X)\cup A)+ r_N(T-Y) 
  \end{equation}
  \begin{equation}\label{dual2}
    |X|+|Y| -r(M)-r(N) + r_M(S-X) +r_N((T-Y)-B) +|(T-Y)\cap B|.
  \end{equation}
  The rank of $X\cup Y$ in $(N^*,M^*;B,A)$ is the minimum of 
  $$r_{N^*}(Y\cup B)+r_{M^*}(X) 
  \qquad\text{ and }\qquad r_{N^*}(Y) + r_{M^*}(X-A) + |X\cap A|,$$
  that is, the minimum of the following two expressions:
  \begin{equation}\label{dual3}
    |Y\cup B| - r(N)+r_N(T-(Y\cup B)) + |X| - r(M) +r_M(S-X)
  \end{equation}
  \begin{equation}\label{dual4}
    |Y|-r(N)+r_N(T-Y) + |X-A|-r(M)+r_M(S-(X-A))+|X\cap A|.
  \end{equation}
  Expressions (\ref{dual1}) and (\ref{dual4}) are equal, as are
  expressions (\ref{dual2}) and (\ref{dual3}), which proves the
  result.
\end{proof}

Proposition \ref{prop:rkdl} and Theorem \ref{thm:dual} motivate the
following question: for a matroid union $K=M^+\join N_0$ as in
Theorem~\ref{unionthm}, under what conditions is $K^*$ the matroid
union of $(M^*)_0$ and an extension of $N^*$?  Note that $K^*$ may not
be such a matroid union.  For example, let $M$ be the uniform matroid
$U_{3,4}$ on $\{a,a',b,b'\}$ and let $N$ be $U_{1,2}$ on $\{c,c'\}$.
Extend $M$ to $M^+$ on $\{a,a',b,b',c,c'\}$ by making $c$ and $c'$
parallel to each other and collinear with $a,a'$ and with $b,b'$.  The
non-spanning circuits of $M^+\join N_0$ are $\{a,a',b,b'\}$,
$\{a,a',c,c'\}$, and $\{b,b',c,c'\}$.  The non-spanning circuits of
$(M^+\join N_0)^*$ are $\{a,a'\}$, $\{b,b'\}$, and $\{c,c'\}$.  It is
easy to see that $(M^+\join N_0)^*$ is not the matroid union of
$(M^*)_0$ and an extension of $N^*$.  This example also shows that the
construction in Theorem \ref{unionthm} does not encompass that in
Theorem \ref{wedgethm} and vice versa.

The first assertion in the next result is a consequence of Lemma
\ref{lem:uniondel} and the definition of principal sums, while the
second part follows from the first and Theorem \ref{thm:dual}.

\begin{cor}\label{cor:prunionmin}
  If $A_0$ is a basis of $M|A$ and $X\subseteq (S-A_0)\cup T$, then
  $$(M,N;A,B)\del X = (M\del (X\cap S),N\del (X\cap T);A-X,B-X).$$
  Dually, if $B_0$ is a basis of $N^*|B$ and $X\subseteq S\cup
  (T-B_0)$, then $$(M,N;A,B)/X = (M/(X\cap S),N/(X\cap T);A-X,B-X).$$
\end{cor}

If too many elements of the set $A$ are deleted, the result might not
be a principal sum, so the hypothesis in Corollary
\ref{cor:prunionmin} is needed.  For example, let $M$ be the parallel
connection, at the point $x$, of two copies of the uniform matroid
$U_{2,4}$, and $N$ be the uniform matroid $U_{1,3}$ on the set
$\{a,b,c\}$.  Set $P=(M,N;\{x\},\{a,b,c\})$.  The principal sum $P$ is
the parallel connection of three $4$-point lines at $x$.  The deletion
$P\del x$ consists of three disjoint $3$-point lines in rank $4$, each
pair of which is coplanar.  It is easy to see that while $P\del x$ is
a semidirect sum of $M\del x$ and $N$, it is not a principal sum of
these matroids.

We next treat the closure operator of a principal sum.  We use
$\mcR_\leq$ as in Lemma~\ref{lem:idealfilter}.

\begin{thm}\label{thm:prclosure}
  Let $P$ be $(M,N;A,B)$.  For $X\subseteq S$ and $Y\subseteq T$,
  $$\cl_P(X\cup Y)  = \left\{
    \begin{array}{ll}
      \cl_M(X\cup A)\cup \cl_N(Y),  & \text{if}\,\, X\cup Y\in \mcR_\leq,\\
      \cl_M(X)\cup \cl_{N\del B}(Y-B)\cup (Y\cap B), & \text{otherwise.}  
    \end{array}\right.  
  $$
\end{thm}

\begin{proof}
  The proof uses the following easy observations.
  \begin{enumerate}
  \item[(i)] For $u\in S$, we have $r_M(X) = r_M(X\cup u)$ if and only
    if $u\in\cl_M(X)$.
  \item[(ii)] For $u\in S$, we have $r_M(X\cup A) = r_M(X\cup A\cup u)$ if
    and only if $u\in\cl_M(X\cup A)$.
  \item[(iii)] For $u\in T$, we have $r_N(Y) = r_N(Y\cup u)$ if and
    only if $u\in\cl_N(Y)$.
  \item[(iv)] If $u\in B-Y$, then $r_N((Y\cup u)-B)+|(Y\cup
    u)\cap B| >r_N(Y-B)+|Y\cap B|$.
  \item[(v)] For $u\in T-B$, the equality $$r_N((Y\cup u)-B)+|(Y\cup
    u)\cap B| = r_N(Y-B)+|Y\cap B|$$ holds if and only if $r_N((Y\cup
    u)-B) = r_N(Y-B)$, that is, $u\in \cl_N(Y-B)-B$.
  \end{enumerate}  
  Observations (ii) and (iii) apply in the first case; the others
  apply in the second.
\end{proof}

Note that if $A\subseteq X$, then $X\cup Y\in\mcR_\leq$.  This
observation is behind the next result, which identifies the flats of a
principal sum.

\begin{cor}\label{cor:prflats}
  The flats of $(M,N;A,B)$ are the following sets:
  \begin{enumerate}
  \item $F_M\cup F_N$ where $F_M\in\mcF(M)$, $A\subseteq F_M$, and
    $F_N\in\mcF(N)$; such flats are in $\mcR_\leq$;
  \item $F_M\cup Y$ where $F_M\in\mcF(M)$, $Y\subseteq T$, $Y-B\in\mcF(
    N\del B)$, and $F_M\cup Y\in\mcR_>$.
  \end{enumerate}
\end{cor}

We next identify the cyclic flats of principal sums.  Recall that a
set $X$ in a matroid $K$ is cyclic if it is a (possibly empty) union
of circuits; equivalently, $r_K(X-a)=r_K(X)$ for all $a\in X$, that
is, $K|X$ has no coloops.  We let $\mcZ(K)$ denote the set of cyclic
flats of $K$.

\begin{thm}\label{thm:prcyclicflats}
  Let $P$ be $(M,N;A,B)$.  The cyclic flats of $P$ are the sets
  \begin{enumerate}
  \item $Z_M\cup Z_N$ where $Z_M\in\mcZ(M)$, $A\not\subseteq Z_M$, and
    $Z_N\in\mcZ(N\del B)$,
  \item $F_M\cup Z_N$ where $F_M\in\mcF(M)$, $A\subseteq F_M$, all
    coloops of the restriction $M|F_M$ are in $A$, $Z_N\in\mcZ(N)$,
    and $Z_N\cap B\ne\emptyset$, and
  \item $Z_M\cup Z_N$ where $Z_M\in\mcZ(M)$, $A\subseteq Z_M$,
    $Z_N\in\mcZ(N)$, and $Z_N\cap B = \emptyset$.
  \end{enumerate}
\end{thm}

\begin{proof}
  First consider the flats $F_M\cup Y$ with $F_M\in\mcF(M)$,
  $Y\subseteq T$, $Y-B\in\mcF(N\del B)$, and $F_M\cup Y\in\mcR_>$.
  Since $\mcR_>$ is an ideal, it follows from Corollary
  \ref{cor:prflats} that $F_M\cup (Y-Y')$ is a flat of $P$ for all
  $Y'\subseteq Y\cap B$, so for $F_M\cup Y$ to be cyclic, we must have
  $Y\subseteq T-B$.  Since $F_M\cup Y\in\mcR_>$, we must have
  $A\not\subseteq F_M$; also, expression (\ref{rankb}) gives the rank
  of $F_M\cup Y$ as well as all of its subsets.  By comparing these
  expressions for $r_P(F_M\cup Y)$, $r_P((F_M-x)\cup Y)$ with $x\in
  F_M$, and $r_P(F_M\cup (Y-y))$ with $y\in Y$, it follows that
  $F_M\cup Y$ is cyclic in $P$ if and only if $F_M\in\mcZ(M)$ and
  $Y\in\mcZ(N\del B)$.  Thus, all such sets $F_M\cup Y$ that are in
  $\mcZ(P)$ are included in item (1) above.  Conversely, the flats
  described in item (1) are in $\mcR_>$ by Corollary
  \ref{cor:prflats}, and, using expression (\ref{rankb}), it is easy
  to check that they are cyclic flats of $P$.

  Now consider the flats $F_M\cup F_N$ with $F_M\in\mcF(M)$,
  $A\subseteq F_M$, and $F_N\in\mcF(N)$.  By Corollary
  \ref{cor:cyclicsets}, if $F_M\cup F_N$ is cyclic in $P$, then it is
  cyclic in $N_0$, so $F_N$ is cyclic in $N$; thus, we restrict our
  attention to flats $F_M\cup F_N$ with $F_N\in\mcZ(N)$.  Examining
  expressions (\ref{ranka}) and (\ref{rankb}) and using the facts that
  $A\subseteq F_M$ and $F_N\in\mcZ(N)$ shows that $F_M\cup
  F_N\in\mcR_<$ if and only if $F_N\cap B \ne \emptyset$, otherwise
  $F_M\cup F_N\in\mcR_=$.  Assume first that $F_M\cup F_N\in\mcR_<$,
  so $F_N\cap B \ne\emptyset$.  Note that for all $x\in F_M$ and $y\in
  F_N$, the sets $(F_M-x)\cup F_N$ and $F_M\cup (F_N-y)$ are in
  $\mcR_\leq$.  An examination of expression (\ref{ranka}) shows that
  $F_M\cup F_N$ is a cyclic flat of $P$ if and only if, besides having
  $F_N\in\mcZ(N)$, all coloops of $M|F_M$ are in $A$.  Now assume that
  $F_M\cup F_N\in\mcR_=$, so $F_N\cap B = \emptyset$.  Thus, for all
  $x\in F_M$ and $y\in F_N$, the sets $(F_M-x)\cup F_N$ and $F_M\cup
  (F_N-y)$ are in $\mcR_\geq$.  An examination of expression
  (\ref{rankb}) shows that $F_M\cup F_N$ is a cyclic flat of $P$ if
  and only if, besides having $F_N\in\mcZ(N)$, we have $F_M\in
  \mcZ(M)$.  Thus, items (2) and (3) describe all such cyclic flats of
  $P$.
\end{proof}

Flats are intersections of hyperplanes (or copoints), so complements
of flats are unions of cocircuits; thus, for a matroid $K$ on $E$, a
subset $X$ of $E$ is cyclic in $K$ if and only if $E-X\in\mcF(K^*)$.
Also, $X$ is a cyclic flat if and only if $X$ is a union of circuits
and $E-X$ is a union of cocircuits. Thus, $\mcZ(K^*) =
\{E-Z\,:\,Z\in\mcZ(K)\}$.  One can check that the sets described in
item (3) of Theorem \ref{thm:prcyclicflats} for $(M,N;A,B)$ are the
complements of those described in the same item for the dual,
$(N^*,M^*;B,A)$, while those described in item (1) for $(M,N;A,B)$ are
the complements of those described in item (2) for $(N^*,M^*;B,A)$,
and vice versa.

It follows from Theorem \ref{thm:prcyclicflats} that in the principal
sum $(M,N;A,B)$, any union or intersection of cyclic flats either
contains $A$ or is disjoint from $B$.  Therefore Theorem
\ref{thm:prcyclicflats} and Corollary \ref{cor:samerkasds} give the
following result.

\begin{cor}\label{cor:prsumtodsum}
  Let $P$ be $(M,N;A,B)$.  If $X$ is a union or intersection of cyclic
  flats of $P$, then $r_P(X) = r_{M\oplus N}(X)$.
\end{cor}

The principal sums in the next corollary are notable in part because
of how Theorem \ref{thm:prcyclicflats} simplifies in this case.

\begin{cor}\label{cor:cycliccase}
  For $A\in\mcZ(M)$ and $B\in\mcZ(N^*)$, the cyclic flats of
  $(M,N;A,B)$ are the sets $Z_M\cup Z_N$ where $Z_M\in\mcZ(M)$,
  $Z_N\in\mcZ(N)$, and either $A\subseteq Z_M$ or $Z_N\cap B =
  \emptyset$.
\end{cor}

The next result describes the circuits of principal sums.

\begin{thm}\label{thm:Ncircuits1}
  A subset $D$ of $S \cup T$ is an $(M,N;A,B)$-circuit if and only if
  one of the following conditions holds:
  \begin{enumerate}
  \item $D\in\mcC(M)$,
  \item $D\in\mcC(N)$ and $D\subseteq T - B$, or
  \item $D = X\cup Y$ where
    \begin{enumerate}
    \item $X\in\mcI(M)$ and no element of $X-A$ is a coloop of
      $M|X\cup A $,
    \item $Y$ is a cyclic set of $N$ and $Y-B\in\mcI(N)$, and
    \item $|X|+|Y|-1 = r_M(X\cup A)+r_N(Y)$.
    \end{enumerate}
  \end{enumerate}
\end{thm}

\begin{proof}
  Let $P$ be $(M,N;A,B)$.  Since $P|S=M$ and $P|(T-B)=N|(T-B)$, sets
  that satisfy either condition (1) or (2) are circuits of $P$.  Now
  assume that condition (3) holds for $X\cup Y$.  We first show that
  condition (c) alone implies that $X\cup Y$ is dependent in $P$.  If,
  to the contrary, $X\cup Y\in\mcI(P)$, then there would be subsets
  $Y_0$ and $Y_1=Y-Y_0$ of $Y$ with $X\cup Y_0$ independent in
  $M^+(A,B)$ and $Y_1$ independent in $N_0$, and so in $N$.  Thus,
  $|X\cup Y_0|\leq r_M(X\cup A)$ by Lemma
  \ref{lem:PrincipalExtension}; also, $|Y_1|\leq r_N(Y)$; adding these
  two inequalities gives a contradiction to condition (c), so $X\cup
  Y$ is dependent in $P$.  We now show that $(X\cup Y)-a\in\mcI(P)$
  for each element $a$ in $X\cup Y$.  First consider $a\in X$.  Since
  $Y-B\in\mcI(N)$, it is contained in a basis $Y_1$ of $N|Y$.  Let
  $Y_0=Y-Y_1$.  Rewrite condition (c) in terms of $Y_0$ and $Y_1$, use
  the equality $r_M((X-a)\cup A) = r_M(X\cup A)$ (which holds since,
  if $a\not\in A$, it is not a coloop of $M|X\cup A$), and simplify:
  we have $$|X-a|+|Y_0|=r_M((X-a)\cup A).$$ Thus, $(X-a)\cup Y_0 \in
  \mcI(M^+(A,B))$ by Lemma \ref{lem:PrincipalExtension}, so $(X-a)\cup
  Y\in\mcI(P)$.  Essentially the same argument applies for $a\in Y$
  (whether $a$ is in the independent set $Y-B$ or in $Y\cap B)$, thus
  completing the proof that the sets given in items (1)--(3) are
  circuits of $P$.

  For the converse, assume that $X\cup Y$ is a circuit of $P$ with
  $X\subseteq S$ and $Y\subseteq T$.  If $Y =\emptyset$, then, since
  $P|S=M$, we have $X\in\mcC(M)$, as described in condition (1).
  Similarly, if $X = \emptyset$ and $Y\subseteq T-B$, then, since
  $P|(T-B)=N|(T-B)$, such circuits are accounted for in condition (2).
  Thus, we may assume that $X\in\mcI(M)$ and $Y-B\in\mcI(N)$.  Extend
  $Y-B$ to a basis $Y_1$ of $N|Y$ and set $Y_0 = Y-Y_1$.  Since $X\cup
  Y$ is dependent in $P$, it must be that $X\cup Y_0$ is dependent in
  $M^+(A,B)$, that is, $|X\cup Y_0|>r_M(X\cup A)$; it follows that
  $|X|+|Y|> r_M(X\cup A)+r_N(Y)$.  However, if $a\in Y$, then $X\cup
  (Y-a)\in\mcI(P)$, which, by Theorem \ref{thm:rank}, gives
  $|X|+|Y-a|\leq r_M(X\cup A)+r_N(Y)$.  The last two inequalities give
  the equality in condition (c).  To show that condition (3) applies,
  note that if some element of $X-A$ were a coloop of $M|X\cup A$, or
  if $Y$ were not cyclic in $N$, then a proper subset of $X\cup Y$
  would satisfy condition (c) and so be dependent, which contradicts
  $X\cup Y$ being a circuit.  Thus, condition (3) applies.
\end{proof}

Using duality or arguments similar to those used above, we can
identify the cyclic sets of principal sums as in the next result.

\begin{thm}\label{thm:cyclicsetsPE} 
  For $X\subseteq S$ and $Y\subseteq T$, the set $X\cup Y$ is cyclic
  in the principal sum $(M,N;A,B)$ if and only if either
  \begin{enumerate}
  \item $X$ is cyclic in $M$, $Y$ is cyclic in $N$, and $Y\cap B =
    \emptyset$, or
  \item the only coloops of $M|X\cup A$ are in $A$, $Y$ is cyclic in
    $N$, and $X\cup Y\in \mcR_<$.
  \end{enumerate}
\end{thm}

We next determine which principal sums are disconnected.  Recall that
the notion of a separator $X$ of a matroid $K$ on $E$ has many
equivalent formulations, including (i) $X$ is a union of components of
$K$, (ii) $r_K(X)+r_K(E-X) = r(K)$, (iii) $K\del X =K/X$, and (iv)
$K=K|X\oplus K\del X$.  Clearly a subset $X$ of $E$ is a separator of
$K$ if and only if $E-X$ is.  It follows that the only semidirect sum
of $M$ and $N$ for which $T$ (equivalently, $S$) is a separator is the
direct sum, $M\oplus N$.  By Theorem \ref{thm:rank}, the rank of $T$
in $(M,N;A,B)$ is the minimum of $r_M(A)+r(N)$ and $r_N(T-B)+|B|$, so,
using the second formulation of the notion of a separator, we get the
following result.

\begin{lemma}\label{lem:chardirsum}
  For a principal sum $(M,N;A,B)$, statements \emph{(1)--(3)} are
  equivalent:
  \begin{enumerate}
  \item $T$ is a separator of $(M,N;A,B)$,
  \item $(M,N;A,B) = M\oplus N$,
  \item either $A$ is a (possibly empty) set of loops of $M$ or $B$ is
    a (possibly empty) set of coloops of $N$.
  \end{enumerate}
\end{lemma}

We now treat connectivity for all principal sums.

\begin{thm}
  Let $P$ be $(M,N;A,B)$.  The principal sum $P$ is disconnected if
  and only if at least one of the following conditions holds:
  \begin{enumerate}
  \item $P= M\oplus N$, 
  \item either $M$ has loops or $N$ has coloops,
  \item $A\ne \emptyset$ and $M$ has a separator $X$ with $A\subseteq
    X\subsetneq S$, or
  \item $B\ne \emptyset$ and $N$ has a nonempty separator $Y$ that is
    disjoint from $B$.
  \end{enumerate}
\end{thm}

\begin{proof}
  We first show that if one of conditions (1)--(4) holds, then $P$ is
  disconnected.  It is easy to see that loops of $M$ are loops of $P$;
  also, coloops of $N$ are coloops of $P$.  Note that duality relates
  conditions (3) and (4), so to complete this part of the proof, we
  show that if condition (3) holds, then $X\cup T$ is a separator of
  $P$.  For this, note that Corollary \ref{cor:samerkasds} gives
  $r_P(X\cup T) = r_M(X)+ r(N)$ and $r_P(S-X)= r_M(S-X)$, so
  \begin{align*}
    r_P(X\cup T) +r_P(S-X) = &\, r(N)+ r_M(X) + r_M(S-X)\\
    = &\, r(N)+r(M),
  \end{align*}
  which is $r(P)$, as we needed to show.

  For the converse, assume that $U$ and $V$ are complementary nonempty
  separators of $P$.  Set $U_S=U\cap S$ and similarly define $U_T$,
  $V_S$, and $V_T$.  By Theorem \ref{thm:rank}, up to interchanging
  $U$ and $V$, one of the following quantities must be $r(M)+r(N)$:
  \begin{enumerate}
  \item[(a)] $r_M(U_S\cup A)+r_N(U_T) +r_M(V_S\cup A)+ r_N(V_T)$,
  \item[(b)] $r_M(U_S)+r_N(U_T-B)+|U\cap B| + r_M(V_S)+
    r_N(V_T-B)+|V\cap B|$,
  \item[(c)] $ r_M(U_S\cup A) +r_N(U_T) + r_M(V_S)+r_N(V_T-B)+|V\cap
    B|$.
  \end{enumerate}

  Semimodularity gives the inequalities $r_M(U_S\cup A)+r_M(V_S\cup
  A)\geq r(M)+r_M(A)$ and $r_N(U_T)+r_N(V_T)\geq r(N)$, so if the
  quantity in option (a) is $r(M)+r(N)$, then $r_M(A)=0$; thus, by
  Lemma~\ref{lem:chardirsum}, condition (1) holds.

  Similarly, if the quantity in option (b) is $r(M)+r(N)$, then we
  must have $$r_N(U_T-B)+ r_N(V_T-B)+|B| = r(N).$$ It follows that all
  element of $B$ are coloops of $N$, so condition (1) holds.

  Finally, assume that the quantity in option (c) is $r(M)+r(N)$.  We
  may assume that conditions (1) and (2) fail, so $r_M(A)>0$ by
  Lemma~\ref{lem:chardirsum}.  Clearly we have $$r_M(U_S\cup A) +
  r_M(V_S)\geq r(M) \quad \text{and} \quad r_N(U_T)+r_N(V_T-B)+|V\cap
  B|\geq r(N),$$ so the assumption about option (c) forces both of
  these inequalities to be equalities.  Thus,
  \begin{enumerate}
  \item[(i)] $r_M(U_S) = r_M(U_S\cup A)$ and $M=M|U_S \oplus M|V_S$,
    and
  \item[(ii)] $V\cap B = \emptyset$ and $N=N|U_T\oplus N|V_T$,
  \end{enumerate}
  where the first assertion in conclusion (ii) holds since elements in
  $V\cap B$ would be coloops of $N$, which we assumed has none.
  Likewise, conclusion (i) implies that $A \subseteq U_S$ since
  elements of $V\cap A$ would be loops of $M$.  At least one of $V_S$
  and $U_T$ is nonempty, for otherwise $V$ would be $T$, which would
  give $P=M\oplus N$.  If $V_S\ne \emptyset$, then condition (3)
  holds; otherwise, both $U_T$ and $V_T$ are nonempty, so condition
  (4) holds.
\end{proof}

Lemma \ref{lem:chardirsum} addresses the problem of when a principal
sum $(M,N;A,B)$ is the direct sum $M\oplus N$.  We now treat the
general problem of when two principal sums are equal.

\begin{thm}\label{thm:charequality}
  If $\cl_M(A) = \cl_M(A')$ and $\cl_{N^*}(B)=\cl_{N^*}(B')$, then the
  principal sums $(M,N;A,B)$ and $(M,N;A',B')$ are equal.  When
  $(M,N;A,B) \ne M\oplus N$, the converse also holds.
\end{thm}

\begin{proof}
  The extension $M^+(A,B)$ adds the elements in $B$ to the flat
  spanned by $A$, so if $\cl_M(A) = \cl_M(A')$, then $(M,N;A,B) =
  (M,N;A',B)$.  The first part of the result follows from this
  observation and Theorem \ref{thm:dual}.

  With the first part and Theorem \ref{thm:dual}, in order to prove
  the second, it suffices to show that if $A\in\mcF(M)$ and
  $B\in\mcF(N^*)$, then $A$ can be recovered from $(M,N;A,B)$.  Since
  $(M,N;A,B)\ne M\oplus N$, Lemma \ref{lem:chardirsum} gives
  $B\ne\cl_{N^*}(\emptyset)$.  The maximum cyclic flat $Z_N$ of $N$
  contains all elements of $T$ except the coloops of $N$, so $Z_N\cap
  B\ne \emptyset$.  Theorem \ref{thm:prcyclicflats} implies that
  $A\cup Z_N$ is the least cyclic flat of $(M,N;A,B)$ of the form
  $F_M\cup Z_N$ with $F_M\in\mcF(M)$.  Thus, we can determine $A\cup
  Z_N$ and hence $A$, as needed.
\end{proof}

Combining the first part of this theorem with
Lemma~\ref{lem:unionspreserveweak} gives the following result, which
gives more ways to express a principal sum as a matroid union.

\begin{cor}
  For $B\subseteq T$, if $B_0$ is a basis of the restriction
  $N^*|\cl_{N^*}(B)$ of the dual of $N$ and $M^+$ is any (not
  necessarily principal) extension of $M$ to $S\cup T$ with
  $$M^+(A,B_0)\leq_w M^+\leq_w  M^+(A,\cl_{N^*}(B)),$$ then $M^+\join
  N_0 = (M,N;A,B)$.
\end{cor}

The next result is an associative law for principal sums. The proof,
which we omit, is a routine computation based on Theorem
\ref{thm:rank}.

\begin{thm}
  Let $M$, $N$, and $K$ be matroids on the disjoint sets $S$, $T$, and
  $U$.  For $A\subseteq S$, $B\subseteq T$, and $C\subseteq U$,
  $$((M,N;A,B),K;A\cup B,C) = (M,(N,K;B,C);A,B\cup C).$$
\end{thm}

\section{Semidirect sums from Higgs lifts}\label{sec:othercon}

In this section, we present a construction of semidirect sums that is
based on Higgs lifts and that gives a different perspective on
principal sums.  We first recall the definition and a few key
properties of Higgs lifts.  See \cite{splice, Brylawski, Strong} for
more information on this operation.

Let $Q$ be a quotient of the matroid $L$ on $E$.  For any integer $i$
with $0\leq i\leq r(L)-r(Q)$, the function $r$ given by
\begin{equation}\label{Higgsrank}
  r(W) = \min\{r_{Q}(W)+i,\,r_{L}(W)\},
\end{equation}
for $W\subseteq E$, is the rank function of a matroid on $E$; this
matroid is the \emph{$i$-th Higgs lift of $Q$ toward $L$} and is
denoted $H^i_{Q,L}$.  The matroid $H^i_{Q,L}$ is the freest (i.e.,
greatest in the weak order) quotient of $L$ that has $Q$ as a quotient
and has rank $r(Q)+i$.  It is useful to extend the range of $i$ by
letting $H^i_{Q,L}$ be $Q$ if $i<0$ and $L$ if $i> r(L)-r(Q)$.

We will use the following two well-known results about the Higgs lift.
Lemma \ref{lem:HLDual} and the first part of Lemma \ref{lem:HLMinor}
follow from equation (\ref{Higgsrank}) by routine computations; the
second part of Lemma \ref{lem:HLMinor} follows from the first part and
Lemma \ref{lem:HLDual}.  Implicit in the statement of Lemma
\ref{lem:HLDual} is the basic result that if $Q$ is a quotient of $L$,
then $L^*$ is a quotient of $Q^*$.

\begin{lemma}\label{lem:HLDual}
  If $Q$ is a quotient of $L$ and $i+j=r(L)-r(Q)$, then
  $$(H^i_{Q,L})^*=H^j_{L^*,Q^*\!}.$$
\end{lemma}

\begin{lemma}\label{lem:HLMinor}
  For any subset $W$ of $E$ and integer $i$, we have 
  $$H^i_{Q,L}|W = H^i_{Q|W,L|W} \qquad \text{ and } \qquad
  H^i_{Q,L}/W =H^{i-k}_{Q/W,L/W},$$ where $k = r_{L}(W) -r_{Q}(W)$.
\end{lemma}

With these lemmas, we can now give another construction of semidirect
sums.

\begin{thm}\label{thm:higgsgivessd}
  Let $M_q$ be a quotient of $M$ with $r(M)-r(M_q) = i$, and $N^l$ be a
  lift of $N$.  Set $Q = M_q\oplus N$ and $L = M\oplus N^l$.
  \begin{enumerate}
  \item The matroid $Q$ is a quotient of $L$.
  \item The Higgs lift $H^i_{Q,L}$ is a semidirect sum of $M$ and
    $N$.
  \end{enumerate}
\end{thm}

\begin{proof}
  The first assertion is easy to check.  For the second, Lemma
  \ref{lem:HLMinor} gives
  $$H^i_{Q,L}|S = H^i_{Q|S,L|S} = H^i_{M_q,M} = M$$
  and, since $r_{L}(S) -r_{Q}(S) = i$,
  \begin{equation*}
    H^i_{Q,L}/S = H^{i-i}_{Q/S,L/S} = H^0_{N,N^l} = N. \qedhere
  \end{equation*}  
\end{proof}

It follows from Lemma \ref{lem:HLDual} that (unlike the constructions
in Theorems \ref{unionthm} and \ref{wedgethm}), the dual of the
semidirect sum constructed in Theorem \ref{thm:higgsgivessd} is
another instance of the same construction.

We turn to a special case.  Let $A$ be a subset of the set $S$ on
which $M$ is defined, and $B$ a subset of the set $T$ on which $N$ is
defined.  Note that $(M/A) \oplus U_{0,A}$ is quotient of $M$ and
$(N\del B) \oplus U_{|B|,B}$ is a lift of $N$.  As in Theorem
\ref{thm:higgsgivessd}, set $Q = (M/A) \oplus U_{0,A}\oplus N$ and
$L=M\oplus (N\del B) \oplus U_{|B|,B}$.  The rank function of the
$r_M(A)$-Higgs lift, $H$, of $Q$ toward $L$ is give as follows: for
$X\subseteq S$ and $Y\subseteq T$,
$$r_H(X\cup Y) = \min\{ r_{M/A}(X-A) + r_N(Y) +r_M(A),\,
r_M(X)+r_N(Y-B)+|Y\cap B|\},$$ that is,
$$r_H(X\cup Y) = \min\{ r_M(X\cup A) + r_N(Y),\,
r_M(X)+r_N(Y-B)+|Y\cap B|\},$$ which, by Theorem \ref{thm:rank}, is
the rank function of the principal sum $(M,N;A,B)$.  Thus, principal
sums are Higgs lifts.  Alternative derivations of the results on
duality, the closure operator, flats, cyclic flats, and circuits for
principal sums can be given by applying the corresponding results
about Higgs lifts (see, e.g., \cite[Theorem 2.3]{splice}).

Following Lemma \ref{lem:SDPindependent}, we used that result to show
that the free product $M\frp N$ is the principal sum $(M,N;S,T)$.
With the remarks in the previous paragraph, it follows that $M\frp N$
is also the $r(M)$-th Higgs lift of $U_{0,S}\oplus N$ toward $M\oplus
U_{|T|,T}$.

We close this section with examples that show that some semidirect
sums given by the construction in Theorem \ref{unionthm} are not given
by that in Theorem \ref{thm:higgsgivessd} and vice versa.

To see that some semidirect sums of the form $M^+\join N_0$ are not
Higgs lifts of the type treated in Theorem \ref{thm:higgsgivessd},
consider the semidirect sum $K$ in Figure \ref{whirl}.  Note that if
$K$ were $H^i_{Q,L}$ for some $Q$ and $L$ as in Theorem
\ref{thm:higgsgivessd}, then $d$, $e$, and $f$ are coloops of $N^l$
(since $K\del S$ is free) and so of $L$.  Nothing in $L$ distinguishes
$d$ from $f$ and these elements are parallel in $N$ and so in $Q$;
hence no Higgs lift can distinguish $d$ from $f$ by having $\{a,f,e\}$
and $\{c,d,e\}$ as lines.  Thus, $K$ does not arise from the
construction in Theorem \ref{thm:higgsgivessd}.

Let $M$ be the rank-$1$ matroid on $x$.  Let $N$ be the truncation of
$U_{1,2}\oplus U_{1,2}\oplus U_{1,2}$ to rank $2$.  Let $K$ be the
free extension of $U_{1,2}\oplus U_{1,2}\oplus U_{1,2}$ by the element
$x$.  Note that $K$ is a semidirect sum of $M$ and $N$, but it is not
a matroid union of the form $M^+\join N_0$ since the only such matroid
union with three pairs of parallel elements is $M\oplus N$.  However,
if $M_q$ is the rank-$0$ matroid on $x$ and $N^l$ is $U_{1,2}\oplus
U_{1,2}\oplus U_{1,2}$, then $K$ is the first Higgs lift of $M_q\oplus
N$ toward $M\oplus N^l$.

\section{Principal sums and transversal matroids}\label{sec:transv}

Well-known results imply that if a semidirect sum of $M$ and $N$ is
representable over a field $\mathbb{F}$, then so are $M$ and $N$;
also, if $M$ and $N$ are $\mathbb{F}$-representable, then any
principal sum of $M$ and $N$ is representable over $\mathbb{F}$ or, if
$\mathbb{F}$ is finite, over a sufficiently large extension of
$\mathbb{F}$.  Theorem \ref{thm:trans} treats similar results for
transversal and fundamental transversal matroids.  Recall that a
\emph{fundamental} (or \emph{principal}) \emph{transversal matroid} is
a transversal matroid that has a presentation by a set system
$(D_1,D_2,\ldots,D_r)$ where each set $D_i$ has at least one element
that is in no set $D_j$ with $j\ne i$.

\begin{thm}\label{thm:trans}
  Let $P$ be the principal sum $(M,N;A,B)$.
  \begin{enumerate}
  \item If $A\in \mcZ(M)$ and both $M$ and $N$ are transversal, then
    $P$ is transversal.
  \item If $A\in\mcZ(M)$, $B\in \mcZ(N^*)$, and both $M$ and $N$ are
    fundamental transversal, then $P$ is fundamental transversal.
  \end{enumerate}
  Let $K$ be a semidirect sum $M^+\join N_0$ of $M$ and $N$ as in
  Theorem \ref{unionthm}.
  \begin{enumerate}
    \addtocounter{enumi}{+2}
  \item If $K$ is transversal, then $M$ and $N$ are transversal.
  \item If $K$ is fundamental transversal, then $M$ and $N$ are
    fundamental transversal.
  \end{enumerate}
\end{thm}

For the free product, Crapo and Schmitt \cite{uft} showed that
statement (1) holds without requiring $S$ to be cyclic, and
counterparts of the other three statements were proven in \cite{ftm}.
The dual of a transversal matroid need not be transversal, but, as Las
Vergnas \cite{mlv} proved, the dual of a fundamental transversal
matroid is fundamental transversal (see also \cite{ftm}).  A matroid
$K$ for which both $K$ and $K^*$ are transversal is sometimes called
\emph{bitransversal}.  It follows from Theorems \ref{thm:dual} and
\ref{thm:trans} that statement (1) holds if ``bitransversal'' replaces
``transversal'' and the hypothesis that $B$ is in $\mcZ(N^*)$ is
added; statement (3) holds for bitransversal matroids with no
additional hypotheses.

The truncation of $U_{1,2}\oplus U_{1,2}\oplus U_{1,2}$ to rank $2$ is
not transversal, yet it is a principal sum of any of its
single-element deletions and a loop, both of which are transversal, so
the hypothesis $A\in \mcZ(M)$ in statement (1) cannot be omitted.  The
principal sum in Figure \ref{prtrans} shows the necessity of the
hypothesis $B\in\mcZ(N^*)$ in statement (2): in that example, both $M$
and $N$ are fundamental transversal, but the principal sum is not
fundamental.  Let $K$ be the free extension of $U_{1,2}\oplus
U_{1,2}\oplus U_{1,2}$ by an element $x$.  Note that $K$ is a
semidirect sum of $K|x$ and $K/x$, and while $K$ is fundamental
transversal, $K/x$ is not transversal.  Thus, the hypothesis that $K$
has the form $M^+\join N_0$ is needed in statements (3) and (4).  This
example and the last paragraph of Section \ref{sec:othercon} also show
that statements (3) and (4) do not extend to the semidirect sums given
by the Higgs lift construction in Theorem \ref{thm:higgsgivessd}.
Along with the following remarks, this example also shows that these
results also do not extend to the construction in Theorem
\ref{wedgethm}: if $M$ is a free matroid, then the semidirect sums of
$M$ and $N$ are the coextensions of $N$ to $S\cup T$ that have rank
$r(N)+|S|$; furthermore, all such coextensions arise from the
construction in Theorem \ref{wedgethm}.

Statement (1) of Theorem \ref{thm:trans} is a consequence of the
following two well-known results: the class of transversal matroids is
closed under principal extensions on cyclic flats (hence, with the
assumptions in statement (1), $M^+(A,B)$ is transversal, as is $N_0$);
also, matroid unions of transversal matroids are transversal.

To prove statements (2)--(4), we will use the following result, the
first part of which is a refinement by Ingleton of a result by Mason
(see \cite{ing}); the second part is from \cite{ftm}, where proofs of
both assertions can be found.  We use $\cup \mcA$ and $\cap \mcA$ to
denote the union and intersection of a family $\mcA$ of sets.

\begin{prop}\label{prop:ftmchar}
  A matroid $K$ is transversal if and only if
  \begin{equation}
    r(\cap \mcA) \leq \sum_{\mcA'\subseteq \mcA}
    (-1)^{|\mcA'|+1}r(\cup \mcA')\label{mi} 
  \end{equation}
  for every nonempty subset $\mcA$ of $\mcZ(K)$.  Also, $K$ is a
  fundamental transversal matroid if and only if equality holds in
  inequality \emph{(\ref{mi})} for every nonempty subset $\mcA$ of
  $\mcZ(K)$.
\end{prop}

We now prove statement (2) of Theorem \ref{thm:trans}.

\begin{proof}[Proof of part (2) of Theorem \ref{thm:trans}]
  It is well-known and easy to prove that direct sums of fundamental
  transversal matroids are fundamental transversal; also,
  $$\mcZ(M\oplus N) = \{Z_M\cup Z_N\,:\,Z_M\in\mcZ(M) \text{ and }
  Z_N\in\mcZ(N)\}.$$ Corollary \ref{cor:cycliccase} gives
  $\mcZ(P)\subseteq \mcZ(M\oplus N)$.  Since inequality (\ref{mi})
  holds with equality in $M\oplus N$, the inclusion $\mcZ(P)\subseteq
  \mcZ(M\oplus N)$ and Corollary \ref{cor:prsumtodsum} imply that
  equality holds in inequality (\ref{mi}) in $P$, so $P$ is
  fundamental transversal by Proposition \ref{prop:ftmchar}.
\end{proof}

The following well-known results, the first two of which are easy to
see, will be used in the proof of statements (3) and (4) of Theorem
\ref{thm:trans}.
\begin{enumerate}
\item[(i)] Let $x$ be a loop or a coloop of $K$.  The matroid $K$ is
  transversal if and only if $K\del x$ is transversal.  The same holds
  for fundamental transversal matroids.
\item[(ii)] Restrictions of transversal matroids are transversal.
\item[(iii)] Contractions of transversal matroids by cyclic sets are
  transversal.  The same holds for fundamental transversal matroids.
\end{enumerate}
With the essential assumption that the sets being contracted are
cyclic, it is not hard to give the required presentations of the
contractions in statement (iii); alternatively, these assertions can
be proven using Proposition \ref{prop:ftmchar} and the observation
that the closure of a cyclic set is a cyclic flat, and for $Z\in
\mcZ(K)$,
$$\mcZ(K/Z) = \{W-Z\,:\,W\in\mcZ(K) \text{ and } Z\subseteq W\}.$$

\begin{proof}[Proof of parts (3) and (4) of Theorem \ref{thm:trans}.]
  To prove part (3), assume that $K$ is transversal.  Statement (ii)
  above implies that $M$ is transversal since $M=K|S$.  To simplify
  the proof that $N$ is transversal, we first reduce it to the case in
  which $M$ is a free matroid.  Let $S_1$ be the largest cyclic flat
  of $M$, so $S_1$ consists of all elements of $S$ other than the
  coloops of $M$.  Since $K|S = M$, the set $S_1$ is cyclic in $K$, so
  $K/S_1$ is transversal by statement (iii) above.  Elements in $S$
  are loops in $N_0$, so Lemma \ref{lem:uniondel} gives
  $$K/S_1 = (M^+/S_1)\join (N_0/S_1),$$ which is also a matroid
  union of the type in Theorem \ref{unionthm}.  The set $S-S_1$ on
  which $M/S_1$ is defined is the set of coloops of $M$, so $M/S_1$ is
  a free matroid.  This justifies the reduction.

  Thus, let $M$ be the free matroid on $S$.  For a circuit $C$ of $N$,
  let $S_C$ be the subset $I(S,C)$ of the basis $S$ of $M$ given by
  equation (\ref{IBC}).  Part (2) of Lemma \ref{lem:Ncircuits} gives
  $C\cup S_C\in\mcC(K)$.  For a cyclic set $Z$ of $N$, let $S_Z$ be
  the union of all sets $S_C$ as $C$ ranges over the circuits of $N$
  that are contained in $Z$.  Thus, $S_Z$ is the minimum subset of
  $S$, relative to inclusion, with $Z\subseteq \cl_{M^+}(S_Z)$.  We
  will use the following results about $S_Z$, which we prove below:
  \begin{itemize}
  \item[(a)] if $Z'$ is also a cyclic set of $N$, then $S_{Z\cup Z'} =
    S_Z\cup S_{Z'}$,
  \item[(b)] $Z\cup S_Z$ is a cyclic set of $K$,
  \item[(c)] if $Z\in\mcZ(N)$, then $Z\cup S_Z\in\mcZ(K)$,
  \item[(d)] all elements of $S-S_Z$ are coloops of $K|(Z\cup S)$,
  \item[(e)] $r_K(Z\cup S_Z) = r_N(Z) +|S_Z|$, and
  \item[(f)\hspace{1pt}] if $Z_1,Z_2,\ldots, Z_t$ are
    cyclic sets of $N$, then
    $$r_K\bigl(\bigcap_{i=1}^t
    (Z_i\cup S_{Z_i})\bigr ) \bigr) = r_N(\bigcap_{i=1}^t Z_i)
    +\bigl|\bigcap_{i=1}^t S_{Z_i}\bigr|.$$
  \end{itemize}
  Property (a) holds by construction, as does property (b) since, as
  noted above, if $C\in\mcC(N)$, then $C\cup S_C\in\mcC(K)$.  For
  property (c), the set $Z\cup S$ is a flat of $N_0$ and
  $\cl_{M^+}(S_Z)$ is a flat of $M^+$ that contains $Z$ and intersects
  $S$ in $S_Z$, so, by Corollary \ref{cor:flats}, their intersection,
  which is $Z\cup S_Z$, is a flat of $K$; this observation and
  property (b) prove property (c).  It is not hard to see that all
  elements of $S-S_Z$ are coloops of $M^+|(Z\cup S)$, so $Z\cup S_Z$
  is the largest cyclic subset of $Z\cup S$ in $M^+$; by Corollary
  \ref{cor:cyclicsets}, it follows that $Z\cup S_Z$ is the largest
  cyclic subset of $Z\cup S$ in $K$, which proves property (d).  Since
  $N = K/S$ and $S$ is independent in $K$, we have $r_N(Z) = r_K(Z\cup
  S) -|S|$; the equality in property (e) follow from this equality and
  property (d).  The equality in part (f) follows similarly since
  elements in $S$ that are not in all $S_{Z_i}$ are coloops of the
  restriction of $K$ to $\bigcap_{i=1}^t (Z_i\cup S_{Z_i})$.

  To show that $N$ is transversal, let $\mcA$ be a nonempty collection
  of cyclic flats of $N$.  For $\mcA'\subseteq \mcA$, let $\mcA'_K =
  \{Z\cup S_Z\,:\,Z\in \mcA'\}$.  Property (c) gives $\mcA_K\subseteq
  \mcZ(K)$.  Since $K$ is transversal, we have, by Proposition
  \ref{prop:ftmchar},
  \begin{equation}\label{KnownMI}
    r_K(\cap \mcA_K) \leq \sum_{\mcA'_K\subseteq \mcA_K}
    (-1)^{|\mcA'_K|+1}r_K(\cup \mcA'_K).
  \end{equation}
  By properties (e) and (a), the right side of the equation can be
  rewritten as
  $$\sum_{\mcA'\subseteq \mcA}
  (-1)^{|\mcA'|+1}r_N(\cup \mcA') + \sum_{\mcA'\subseteq \mcA}
  (-1)^{|\mcA'|+1} \bigl|\bigcup_{Z\in\mcA'}S_Z\bigr|,$$ which, by
  inclusion/exclusion, is
  \begin{equation}\label{MIa}
    \sum_{\mcA'\subseteq \mcA} (-1)^{|\mcA'|+1}r_N(\cup \mcA') +
    \bigl|\bigcap_{Z\in\mcA}S_Z\bigr|.
  \end{equation}
  By property (f), the right side of inequality (\ref{KnownMI}) is
  equal to
  \begin{equation}\label{MIb}
    r_N(\cap \mcA) +   \bigl|\bigcap_{Z\in\mcA}S_Z\bigr|. 
  \end{equation}
  Thus, expression (\ref{MIb}) is no greater than expression
  (\ref{MIa}); canceling the common term gives inequality (\ref{mi})
  for the collection $\mcA$ of cyclic flats of $N$.  Thus, by
  Proposition \ref{prop:ftmchar}, $N$ is transversal.

  For part (4), note that since the class of fundamental transversal
  matroids is closed under duality, it suffices to show that if $K$ is
  fundamental transversal, then so is $N$.  The proof that $N$ is
  fundamental transversal follows from the same type of argument as
  above, but using the second part of Proposition \ref{prop:ftmchar}.
\end{proof}

\end{document}